\documentclass[11pt]{amsart}
\usepackage{amssymb}
\usepackage{verbatim}
\usepackage{amsmath}
\usepackage{}
\usepackage{amssymb,latexsym}
\usepackage{enumerate}
\newcommand{\op} {\overline{\partial}}
\newcommand{\dbar}{\ensuremath{\overline\partial}}

\makeatletter
\newcommand{\sumprime}{\if@display\sideset{}{'}\sum%
            \else\sum'\fi}
\makeatother

\begin{document}

\numberwithin{equation}{section}

\newtheorem{theorem}{Theorem}[section]
\newtheorem{proposition}[theorem]{Proposition}
\newtheorem{conjecture}[theorem]{Conjecture}
\def\theconjecture{\unskip}
\newtheorem{corollary}[theorem]{Corollary}
\newtheorem{lemma}[theorem]{Lemma}
\newtheorem{observation}[theorem]{Observation}
\newtheorem{definition}{Definition}
\numberwithin{definition}{section} 
\newtheorem{remark}{Remark}
\def\theremark{\unskip}
\newtheorem{kl}{Key Lemma}
\def\thekl{\unskip}
\newtheorem{question}{Question}
\def\thequestion{\unskip}
\newtheorem{example}{Example}
\def\theexample{\unskip}
\newtheorem{problem}{Problem}

\thanks{Research supported by Knut and Alice Wallenberg Foundation, and the China Postdoctoral Science Foundation.}

\address{School of Mathematical Sciences, Fudan University, Shanghai, 200433, China}
\address{Current address: Department of Mathematical Sciences, Chalmers University of Technology and
University of Gothenburg. SE-412 96 Gothenburg, Sweden}
\email{wangxu1113@gmail.com}
\email{xuwa@chalmers.se}

\title[Relative $\dbar$-complex]{Relative $\dbar$-complex and its curvature properties}
 \author{Xu Wang}
\date{\today}

\begin{abstract} We shall define the relative $\dbar$-complex and study the curvature properties of the associated vector bundles. As an application, we shall prove that Yamaguchi's theory on subharmonicity of the Green operator can be seen as a curvature property of the quotient bundle. A short survey of other recent applications will also be given in this paper.

\bigskip

 \noindent{{\sc Mathematics Subject Classification} (2010): 32A25}

\smallskip

\noindent{{\sc Keywords}: $\dbar$-operator, subbundle, quotient bundle, H\"ormander $L^2$-theory, Yamaguchi's theory, Berndtsson's complex Brunn-Minkowski theory.}
\end{abstract}
\maketitle


\section{Introduction}

Our motivation to write this paper is to give a unified proof of  recent results in \cite{BPW16}, \cite{Wang15}, \cite{Wang16} and \cite{Wang16-1} by using the relative $\dbar$-complex. Let $\pi$ be a \textbf{holomorphic submersion} from a complex manifold $\mathcal X$ to a complex manifold $B$, let  $E$ be a holomorphic vector bundle over $\mathcal X$, let $X_t$ be the fibre at $t$ and $E_t$ be the restriction of $E$ to $X_t$. Denote by $\mathcal A^{p,q}_t$ the space of smooth $E_t$-valued $(p,q)$-forms on $X_t$. 
We call the collection of $\dbar$-operators on fibres
$$ \{\dbar^t\}_{t\in B} : \{\mathcal A^{p,q}_t\}_{t\in B} \to \{\mathcal A_t^{p,q+1}\}_{t\in B}, $$
the \textbf{relative $\dbar$-complex}. Put $\mathcal A^{p,q}= \{\mathcal A^{p,q}_t\}_{t\in B}$, $\mathcal K^{p,q}:= \{ {\rm Ker}~ \dbar^t\}_{t\in B}, \  \mathcal I^{p,q+1}:= \{{\rm Im}~ \dbar^t\}_{t\in B}$. Then we have the following exact sequence
\begin{equation}\label{eq:exact}
0\to \mathcal K^{p,q} \to \mathcal A^{p,q} \to  \mathcal I^{p,q+1} \to 0.
\end{equation}
We call $\mathcal K^{p,q}$ the \textbf{$\dbar$ subbundle} and $\mathcal I^{p,q+1}$ the \textbf{$\dbar$ quotient bundle}.
In general, they are not holomorphic vector bundles but still we can define the Chern connection on them as an \textbf{operator on the space of smooth sections} (see \cite{Bern09} and \cite{LS14} for related results). Our starting point is the following result of Berndtsson (see Theorem 1.1 in \cite{Bern09}):

\medskip

\emph{If $\pi$ a product family then one may study the curvature properties of $\mathcal K^{n,0}$, $n$ is the fibre dimension, by looking at $\mathcal K^{n,0}$ as a holomorphic subbundle of $\mathcal A^{n,0}$.}

\medskip

The curvature properties of $\mathcal K^{n,0}$ are crucial in Berndtsson's \textbf{complex Brunn-Minkowski theory}, which also contains curvature properties of vector bundles associated to a non-trivial fibration (see \cite{Bern06}, \cite{Bern09}, \cite{Bern11}, \cite{Bern09a}, \cite{Bern13}, \cite{Bern14}, \cite{BL14}, \cite{BB14}, \cite{BernPaun08}, see also \cite{Tsuji05}, \cite{Sch12}, \cite{LiuYang13}, \cite{GS15}, \cite{MT07}, \cite{MT08} and \cite{Wang15} for other generalizations). Thus it is natural to ask: \textbf{Whether the relative $\dbar$-complex can be used to prove the curvature properties of the associated vector bundles for a general fibration ?}  We shall try to answer this question in this paper.
In particular, we shall show that  the relative $\dbar$-complex can be used to give a unified proof of recent results in \cite{BPW16}, \cite{Wang15}, \cite{Wang16} and \cite{Wang16-1}.

Another related theory is \textbf{Yamaguchi's theory} on subharmonicity properties of the Green-operators (see \cite{Yamaguchi89}, \cite{Maitani84}, \cite{MY04}, \cite{KLY11} and \cite{Wang16}). In the KAWA-NORDAN 2014 conference in Marseille, Levenberg asked the following question:

\medskip

\emph{Is Yamaguchi's theory a curvature property ?}

\medskip

In this paper, we shall answer Levenberg's question by proving that \textbf{Yamaguchi's theory can be seen as a curvature property of the quotient bundle $\mathcal I^{n,n}$}. 

\section{Family of domains in a fixed manifold}

First we shall consider variation of domains $X_t$, $t\in B$, in a \textbf{fixed} complex manifold $M$. We shall show how to use the \textbf{curvature properties of the $\dbar$-quotient bundle} to study \textbf{variation of the $L^2$-minimal solution of the $\dbar$-operator on fibres}, see \cite{Wang16}.

\medskip

Let $(E_0, h^0)$ be a \textbf{fixed} holomorphic vector bundle over $(M,\omega^0)$, let $g$ be a \textbf{fixed} smooth $\dbar$-closed $E_0$-valued $(n,q+1)$-form on $M$. We shall introduce the following definition:

\begin{definition} If $a^t$ is the \textbf{$L^2$-minimal solution} of $\dbar^t(a^t)=g$ on $X_t$ with respect to $\omega^0$ and $h^0$ then we call  $||a^t||$ the \textbf{Green norm} of $g$ on $X_t$ and denote it by $||g||_G(t)$. 
\end{definition}

\textbf{Remark 1}: If the above $\dbar$-equation has no $L^2$-solution then we say the Green norm of $g$ is infinite on $X_t$.

\medskip

\textbf{Relation with Green operator}: Since the $L^2$-minimal solution $a^t$ of $\dbar^t(a^t)=g$ can be written as $a^t=(\dbar^t)^* G^t g$, where each $G^t$ is the \textbf{Green-operator} with respect to the $\dbar^t$-Laplacian on $X_t$. Thus we have
\begin{equation*}
||g||^2_G(t)=||a^t||^2=((\dbar^t)^* G^t g, a^t)=(G^t g, g).
\end{equation*}
That is why we call the $L^2$-norm of the minimal solution  of the $\dbar$-equation the Green-norm. See \cite{Bern06} for the \textbf{relation between the Green-norm and the Robin constant}.

\medskip

\textbf{Relation with the quotient norm}: Since the minimal solution is just the \textbf{minimal lift} w.r.t. the \textbf{relative $\dbar$-complex}: $\dbar^t: \mathcal A^{p,q}_t \twoheadrightarrow \mathcal I^{p,q+1}_t$, we know that \textbf{the Green norm is just the quotient norm}. 

\medskip

Now we know that one may use curvature formula of the quotient bundle to study variation of the Green norm. The latter is known as \textbf{Yamaguchi's theory}, thus \textbf{Yamaguchi's theory can be seen as a curvature property of the quotient bundle}, this answers the previous question raised by Levenberg. Now we have another formulation of Yamaguchi's theory:

\begin{theorem}[Developed by Yamaguchi-Maitani-Levenberg-Kim, etc] Under some convexity and curvature assumptions, $-||g||^2_G$ is plurisubharmonic if $g$ is a holomorphic section of the quotient bundle $\mathcal I^{n,n}$, where $n$ is the fibre dimension.
\end{theorem}

Our first result is a generalization of Yamaguchi's theory to the $(n,q+1)$-case. More precisely, we shall prove the following theorem (see the main theorem in \cite{Wang16}):

\begin{theorem}\label{th:Y}[$I^{n,q+1}$-version of Yamaguchi's theorem] Let $g$ be a $\dbar$-closed $E^0$-valued $(n,q+1)$-form on $M$. Assume that the total space $\mathcal X:=\{(z,t)\in M\times B: z\in X_t\}$ is Stein, $\omega^0$ is K\"ahler and 
\begin{equation} \label{eq:q-positive}
i\Theta(E_0,h^0)\wedge (\omega^0)^{q}\geq 0,
\end{equation} 
on $M$. Assume further $g$ has \textbf{compact support in each fibre} $X_t$. Then $-||g||^2_G$ is plurisubarmonic on $B$. 
\end{theorem}

\textbf{Main tools in the proof}: H\"ormander's $L^2$-theory \cite{Hormander-65} and the curvature formula for the $\dbar$ quotient bundle $I^{n,q+1}$. By a generalized version of Berndtsson's approximation process (see \cite{Bern06} or \cite{Wang16} for details), it suffices to prove the following \textbf{product case} with a non-product weight $\psi$ (a real smooth function on $M\times B$):
   
\begin{theorem}[Product case with non-product weight $\psi$] Assume that  $\mathcal X=X_0 \times B$, where 
 $X_0$ is a smoothly bounded strictly pseudoconvex domain in $M$. Assume that the pull back to the total space, say $\omega$, of $\omega^0$ is K\"ahler and
\begin{equation*}
 i\Theta(E,h)\wedge \omega^q > 0,  \ \ E:=E_0 \times B, \ h:=e^{-\psi} (h^0\times B).
\end{equation*}  
Assume further that $g$ has \textbf{compact support} in $X_0$ and \textbf{$\psi$ does not depend on $t$} on ${\rm Supp}(g) \times B$. Then $-||g||^2_G$ is plurisubharmonic on $B$.
\end{theorem}

\begin{proof}

Since the \textbf{Green norm is just the quotient norm}, it is enough to estimate the curvature of the quotient bundle $\mathcal I^{n,q+1}$.  

Let us assume that \textbf{$B$ is one dimensional}. Denote by $\Theta_{t\bar t}$ the curvature operators on $\mathcal A^{n,q}$, then  we have
\begin{equation*}
\Theta_{t\bar t} =[D_{t}, \dbar_{t}]=\psi_{t\bar t}, \ D_{t}:=\partial / \partial t-\psi_t, \ \dbar_{t}:= \partial/\partial \bar t.
\end{equation*}
Since $\mathcal I^{n,q+1}=\mathcal A^{n,q}/\mathcal K^{n,q}$ is the quotient bundle of $\mathcal A^{n,q}$, we have
\begin{equation*}
(\Theta_{t\bar t}^{\mathcal I} g, g )_G=( \psi_{t\bar t} a, a) + ||P(a_{\bar t})||^2, \ \ a: t\mapsto a^t\in \mathcal A_t^{n,q},
\end{equation*}
where $\Theta_{t\bar t}^{\mathcal I}$ is the curvature operator on $\mathcal I^{n,q+1}$, each $a^t$ is the \textbf{$L^2$-minimal solution} of $\dbar^t(\cdot)=g$ and $P$ denotes the \textbf{orthogonal projection} to $\mathcal K^{n,q}$. Notice that  \textbf{Hamilton's theory} (see \cite{Hamilton77} and \cite{Hamilton79}) implies that $a$ is a smooth section.

Since $g$ does not depend on $t$ and has \textbf{compact support} in each fibre, we know that there exists a \textbf{fixed} smooth $L^2$-solution, say $u$, of $\dbar u=g$ on $X_0$. 
Since $u$ is fixed, we have $u_{\bar t}=0$, thus $u$ is a \textbf{holomorphic section} of $\mathcal A^{n,q}$. Since \textbf{$g$ is the image of $u$ under the $\dbar$-quotient map}, we know that  $g$ is a \textbf{holomorphic section} of $\mathcal I^{n,q+1}$, thus we have
\begin{equation*}
(||g||^2_G)_{t\bar t}=||D_{t} a||^2- (\Theta_{t\bar t}^{\mathcal I} g, g )_G \leq ||D_t a||^2- (\psi_{t\bar t}a,a).
\end{equation*}
We shall use \textbf{H\"ormander's $L^2$-theory to control the norm} of $D_t a$. Notice that each $D_t a$ is the \textbf{$L^2$-minimal solution} of
\begin{equation*}
\dbar^t(\cdot)= \dbar^t D_t a=[\dbar^t, D_t] a+D_t g=[\dbar^t, D_t] a =-\dbar^t\psi_t \wedge a,
\end{equation*}
the \textbf{third} equality follows from $D_t g=g_t-\psi_t g \equiv 0$ since $\psi$ \textbf{does not depend on $t$ on} ${\rm Supp}(g) \times B$. By \textbf{H\"ormander's $\dbar$-$L^2$-estimate}, we have
\begin{equation*}
||D_t a||^2 \leq (Q^{-1}\dbar^t\psi_t \wedge a, \dbar^t\psi_t \wedge a),
\end{equation*}
where
\begin{equation*}
Q:=[i\Theta(E_0, e^{-\psi^t}h^0), \Lambda_{\omega^0}].
\end{equation*}
Thus
\begin{equation*}
(||g||^2_G)_{t\bar t} \leq (Q^{-1}\dbar^t\psi_t \wedge a, \dbar^t\psi_t \wedge a)- (\psi_{t\bar t}a,a)
\end{equation*}
By a direct computation, we know that $i\Theta(E,h)\wedge \omega^q > 0$ implies
\begin{equation*}
(Q^{-1}\dbar^t\psi_t \wedge a, \dbar^t\psi_t \wedge a) \leq (\psi_{t\bar t}a,a).
\end{equation*}
Thus $(||g||^2_G)_{t\bar t} \leq 0$. The proof is complete.
\end{proof}

If we consider the \textbf{subbundle} $\mathcal K^{n,q}$ and use \textbf{H\"ormander's $L^2$-theory to control the second fundamental form} as Berndtsson did for the $q=0$ case (see the proof of Theorem 1.1 in \cite{Bern09}) then we can prove the following theorem in \cite{Wang16}:

\begin{theorem}\label{th:B}[$(n,q)$-version of Berndtsson's theorm] Let $v$ be a \textbf{fixed} smooth $E_0$-valued $(n,q)$-form with \textbf{compact support} in each fibre. Assume that $\omega^0$ is \textbf{K\"ahler}, the total space is \textbf{Stein} and $i\Theta(E_0,h^0)\wedge (\omega^0)^{q}\geq 0$, then $\log||P(v)||: t\mapsto ||P^t(v)||$ is  plurisubharmonic on $B$, where each $P^t(v)$ denotes the Bergman projection of $v$ to $\ker \dbar^t=\mathcal K^{n,q}_t$.
\end{theorem}

\textbf{Relation with the Ohsawa-Takegoshi theorem}: If $q=0$ then by Berndtsson-Lempert-Blocki's method \cite{BL14}, the above theorem (with $v$ a general current, see section 5.3 in \cite{Wang15} or \cite{Wang16}) can be used to prove Blocki-Guan-Zhou's sharp version of the Ohsawa-Takegoshi theorem (see \cite{OT87}, \cite{Chen11}, \cite{Blocki13} and \cite{GuanZ15}, to cite just a few).

\section{Proper K\"ahler fibration}

Now let us consider the case that $\pi: \mathcal X\to B$ is a \textbf{proper} fibration. We shall prove that:

\begin{theorem}[Generalized Berndtsson-Mourougane-Takayama's theorem]
Assume that $\omega$ is K\"ahler and $i\Theta(E, h)\wedge \omega^q \geq 0$. Assume further that $\dim H^{n,q}(E_t)$ is a \textbf{constant}. Then
\begin{itemize}
\item $\mathcal K^{n,q}$ is \textbf{Nakano semipositive} if $q=0$ or $c_{j\bar k}(\omega)\equiv 0$;

\item $R^{q} \pi_*\mathcal O(\mathcal K_{\mathcal X/B} \otimes E) \simeq \mathcal K^{n,q}/\mathcal I^{n,q}$ is \textbf{Griffiths semipositive},
\end{itemize}
where $c_{j\bar k}(\omega):=\langle V_j, V_k\rangle_{\omega}$ and each $V_j$ denotes the horizontal lift of $\partial/\partial t^j$ with respect to $\omega$.
\end{theorem}

\textbf{Remark}: In case $\mathbf E$ \textbf{is Nakano-semipositive}, Mourougane-Takayama \cite{MT08} proved that $\mathcal K^{n,q}/\mathcal I^{n,q}$ is \textbf{Nakano semipositive} (with different metric, i.e. not the quotient metric).

\begin{proof}

Recall in the \textbf{product case}, the Chern connection on $\mathcal A^{n,q}$ is defined by $\dbar_{t^j}=\partial/\partial t^j$ and $D_{t^j}=\partial/\partial t^j-\psi_j$. But now we have to use the \textbf{generalized Lie derivatives} to define the Chern connection on $A^{n,q}$, i.e.,
\begin{equation*}
D_{t^j} \mathbf u:=[\partial^E, \delta_{V_j}]\mathbf u, \ \dbar_{t^j}\mathbf u=[\dbar, \delta_{\bar V_j}]\mathbf u, \ \Theta_{j\bar k}:=[D_{t^j}, \dbar_{t^k}],
\end{equation*}
where $d^E:=\dbar+\partial^E$ denotes the Chern connection on $E$, each $V_j$ is the \textbf{horizontal lift} of $\partial/\partial t^j$ with respect to $\omega$ and $\mathbf u$ is the \textbf{representative} of a smooth section $u: t\mapsto u^t$ of $\mathcal A^{n,q}$, i.e. $\mathbf u|_{X_t}=u^t$. Since $\mathcal K^{n,q}$ is the \textbf{subbundle} of $\mathcal A^{n,q}$, we have
\begin{equation}\label{eq:theta-k-forget}
(\Theta_{j\bar k}^{\mathcal K} u, v )=(\Theta_{j\bar k}u, v)-(P^{\bot} D_{t^j} u, P^{\bot} D_{t^k} v).
\end{equation} 
By Theorem 2.5 in \cite{BPW16},  we have
\begin{equation}\label{eq:curvature-A}
(\Theta_{j\bar k}u, v)=([L_j, L_{\bar k}]u, v)-(\partial \bar V_{k}|_{X_t} \lrcorner~ u, \partial \bar V_{j}|_{X_t} \lrcorner~ v)+ (\dbar V_{j}|_{X_t} \lrcorner~ u, \dbar V_{k}|_{X_t} \lrcorner~ v).
\end{equation}
By Proposition 4.2 in \cite{Wang15}, we know that
\begin{equation}\label{eq:LJK-Lie}
 [L_j,L_{\bar k}]= [d^{E}, \delta_{[V_j,\bar V_k]}]+ \Theta(E,h)(V_j,\bar V_k).
\end{equation}
Moreover, by Lemma 6.1 in \cite{Wang15}, we know that
\begin{equation}
\delta_{[V_j,\bar V_k]} =(\partial^t c_{k\bar j}(\omega))^*-(\dbar^t c_{k\bar j}(\omega))^*.
\end{equation}
Notice that $\alpha^*=*\bar\alpha*$ if $\alpha$ is an one-form, by a direct computation, we have
\begin{eqnarray}
 \label{eq:new} ([d^{E}, \delta_{[V_j,\bar V_k]}]u, v) & = & (c_{j\bar k}(\omega) u, \dbar^t(\dbar^t)^*v) -(c_{j\bar k}(\omega) u, \partial^{E_t}(\partial^{E_t})^*v) \\
 \nonumber &  & + ~ (c_{j\bar k}(\omega)(\partial^{E_t})^*u,(\partial^{E_t})^*v  ) -  (c_{j\bar k}(\omega)(\dbar^t)^*u,(\dbar^t)^*v  ) \\
  \nonumber &  & + ~ (c_{j\bar k}(\omega)(\partial^{E_t})^*\partial^{E_t} u,v  ) -  (c_{j\bar k}(\omega)(\dbar^t)^*\dbar^t u, v  ) \\
  \nonumber &  & + ~ (c_{j\bar k}(\omega)\dbar^t u,\dbar^t v  ) -  (c_{j\bar k}(\omega)\partial^{E_t} u,\partial^{E_t} v  ).
\end{eqnarray}
Thus if $u$ is an $(n,q)$-form and $c_{j\bar k}(\omega) \equiv 0$ then we have
\begin{equation}\label{eq:a-positive}
(\Theta_{j\bar k}u, v)=(\Theta(E,h)(V_j,\bar V_k)u, v)+ (\dbar V_{j}|_{X_t} \lrcorner~ u, \dbar V_{k}|_{X_t} \lrcorner~ v).
\end{equation} 
Now let us \textbf{control the second fundamental form} $(P^{\bot} D_{t^j} u, P^{\bot} D_{t^k} v)$. Since each $P^{\bot} D_j u$ is the $L^2$-minimal solution of 
\begin{equation}
\dbar^t( \cdot) =\dbar^t D_{t^j} u,
\end{equation}
similar as the product case, one may also use the H\"ormander's $L^2$-theory to control the norm of $P^{\bot} D_j u$. By definition, we have 
\begin{equation}
D_{t^j}\mathbf u=[\partial^E, \delta_{V_j}] \mathbf u.
\end{equation}
Since
\begin{equation}
[\dbar, [\partial^E, \delta_{V_j}]]+[\partial^E, [\dbar, \delta_{V_j}]]+[\delta_{V_j}, [\dbar, \partial^E]]=0,
\end{equation}
we have
\begin{equation}\label{eq:abc-1}
\dbar^t D_{t^j} u=- [\partial^{E_t}, \dbar V_j|_{X_t}] u -(V_j\lrcorner~\Theta(E,h))|_{X_t} \wedge u.
\end{equation}
If each $\partial^{E_t} u_j=0$ then we know that $a:=-\sum P^{\bot} D_j u_j$ is the $L^2$-minimal solution of 
\begin{equation}
\dbar^t( \cdot) = \partial^{E_t} b+ c ,
\end{equation}
where
\begin{equation}
b:= \sum \dbar V_j|_{X_t} \lrcorner ~ u_j, \  c:=\sum (V_j\lrcorner~\Theta(E,h))|_{X_t} \wedge u_j.
\end{equation}
By H\"ormander's $L^2$-theory, $i\Theta(E_t, h^t)\wedge (\omega^t)^q\geq 0$ implies that
\begin{equation}
||a||^2\leq ||b||^2+ \lim_{\varepsilon\to 0} (Q_\varepsilon^{-1}c, c), \ Q_\varepsilon:=[i\Theta(E_t, h^t), \Lambda_{\omega^t}] +\varepsilon.
\end{equation}
Thus we have
\begin{equation}
\sum (\Theta^{\mathcal K}_{j\bar k}u_j, u_k) \geq \sum (\Theta(E,h)(V_j,\bar V_k)u_j, u_k)-\lim_{\varepsilon\to 0} (Q_\varepsilon^{-1}c, c).
\end{equation}
By Lemma 3.10 in \cite{Wang16-1}, we know that the right hand side is non-negative. Similar proof works for other parts of this theorem, please see \cite{Wang16-1} for the details. 
\end{proof}

By a similar argument, one may also prove the following: 

\begin{theorem}[Yamaguchi's theorem for a proper K\"ahler fibration]\label{th:I-curvature} Assume that $\omega$ is K\"ahler, $c_{j\bar k}(\omega)\equiv 0$ and $i\Theta(E, h)\wedge \omega^q \geq 0$. Assume further that the dimension of $H^{n,q}(E_t)$ is a constant. If $g$ is a holomorphic section of the quotient bundle $\mathcal I^{n,q+1}$, and
\begin{equation}
(D_{t^j}g) (t)=([\partial^E, \delta_{V_j}] \mathbf g )|_{X_t}\equiv 0.
\end{equation}
Then $-||g||^2_G$ is a smooth plurisubharmonic function on $B$.
\end{theorem}

\section{Twisted version of Griffiths' theorem}

We will give a short account of a recent joint work with Berndtsson and Paun. We shall show how to look at it by using the relative $\dbar$-complex. Let $\pi: \mathcal X\to B$ be a proper fibration, let $E$ be a holomorphic vector bundle over the total space $\mathcal X$. By Theorem 2.3 in \cite{BPW16} or \cite{Wang16-1}, if $\dim H^{p,q}(E_t)$ does not depend on $t$ then one may look at $$\mathcal H^{p,q}:=R^{q} \pi_*\mathcal O(\wedge^{p} T^*_{\mathcal X/B} \otimes E),$$ as the holomorphic quotient bundle $\mathcal K^{p,q}/\mathcal I^{p,q}$. Notice that the quotient norm of a class $[u^t]$ in $\mathcal K_t^{p,q}/\mathcal I_t^{p,q}$ is equal to
\begin{equation}
||[u^t]||:=\inf \{||u^t+v^t||: v^t\in \mathcal I_t^{p,q}\}.
\end{equation} 
By the Hodge theory, we know that 
\begin{equation}
\inf \{||u^t+v^t||: v^t\in \mathcal I_t^{p,q}\}=||\mathbb H u^t ||,
\end{equation} 
where $\mathbb H u^t$ denotes the $\dbar^t$-harmonic part of $u^t$. We shall also write
\begin{equation}
\mathbb H^{\bot} u^t:= u^t-\mathbb H u^t.
\end{equation} 
Let us denote by $\Theta^{\mathcal H}_{j\bar k}$ the curvature operators on $\mathcal H^{p,q}$. By the curvature formula for the quotient bundle, we have:

\begin{theorem}\label{th:curvature-h} Assume that the total space $\mathcal X$ is K\"ahler. Assume further that $\dim H^{p,q}(E_t)$ does not depend on $t$. Then we have
\begin{equation}\label{eq:curvature-h}
(\Theta^{\mathcal H}_{j\bar k}[u], [v])=(\Theta^{\mathcal K}_{j\bar k}\mathbb H u, \mathbb H v)+ \left(\mathbb H^{\bot} (\dbar_{t^k}\mathbb Hu),\mathbb H^{\bot} (\dbar_{t^j}\mathbb Hv) \right),
\end{equation}
where $[u],[v]$ are smooth sections of $\mathcal H^{p,q}$.
\end{theorem}

Now let us consider the following special cases:
\begin{itemize}
\item $\mathbf A$: $\omega$ is K\"ahler and $\Theta(E, h) \equiv 0$;
\item $\mathbf B$: $p+q=n$, $E$ is a line bundle and $i\Theta(E, h)=\pm \omega$.
\end{itemize}
By the Hodge theory, we know that in both cases, the space of $\dbar^t$-harmonic $E_t$-valued  $(p,q)$-forms is equal to the space of $\partial^{E_t}$-harmonic $E_t$-valued $(p,q)$-forms.
Thus 
\begin{equation}
(d^E \mathbb H u)|_{X_t} \equiv 0,
\end{equation}
which implies that
\begin{equation}
([d^E, \delta_{[V_j, \bar V_{k}]}] \mathbb H u, \mathbb H v)\equiv 0.
\end{equation}
Thus by \eqref{eq:LJK-Lie}, we have
\begin{equation}
([L_j , L_{\bar k}] \mathbb H u, \mathbb H v)\equiv (\Theta(E,h)(V_j, \bar V_k) u, v).
\end{equation}
Assume further that $B$ is one dimensional. Put
\begin{equation}
\Theta^{\mathcal H}_{t\bar t}:=\Theta^{\mathcal H}_{1\bar 1}, \
 \dbar_t:=\dbar_{t^1}, \ D_t:= D_{t^1}, \ V:=V_1, \ A:= (\Theta(E,h)(V, \bar V) u, u).
\end{equation}
By  \eqref{eq:curvature-A}, \eqref{eq:theta-k-forget} and Theorem \ref{th:curvature-h}, we have
\begin{equation}\label{eq:curvature-Final}
(\Theta^{\mathcal H}_{t\bar t}[u], [u])= A +||\dbar V|_{X_t} \lrcorner ~\mathbb Hu||^2- ||\partial \bar V|_{X_t} \lrcorner~ \mathbb Hu||^2+ ||\mathbb H^{\bot} (\dbar_{t}\mathbb Hu)||^2-  ||P^{\bot} (D_{t}\mathbb Hu)||^2
\end{equation}

We shall use the following proposition in \cite{BPW16}:

\begin{proposition}\label{pr:dbar} Assume that $\mathbf A$ or $\mathbf B$ is true. Then $P^{\bot} (D_{t}\mathbb Hu)$ is the $L^2$-minimal solution of $\dbar^t(\cdot)=-\partial^{E_t} (\dbar V|_{X_t} \lrcorner ~\mathbb Hu)$ and $\mathbb H^{\bot} (\dbar_{t}\mathbb Hu)$ is the $L^2$-minimal solution of $\partial^{E_t}(\cdot)= -\dbar^t(\partial \bar V|_{X_t} \lrcorner~ \mathbb Hu)$.
\end{proposition}

Let us denote by $\square$ the $\dbar^t$-Laplace and denote by $\overline{\square}$ the $\partial^{E_t}$-Laplace. By the Bochner-Kodaira-Nakano formula, in case $\mathbf A$, we have $\square \alpha=\overline{\square} \alpha$,  in case  $\mathbf B$, we have $\square \alpha=\overline{\square}\alpha \pm \alpha$ if $i\Theta(E, h)=\pm \omega$. Thus \eqref{eq:curvature-Final}, Proposition \ref{pr:dbar} and page 15 in \cite{Bern11} together imply the following twisted version of Griffiths' theorem in \cite{BPW16}:

\begin{theorem}[Twisted version of Griffiths' theorem]\label{th:direct-image} Assume that $\omega$ is K\"ahler and $\dim H^{p,q}(E_t)$ does not depend on $t$. Assume further that $B$ is one dimenional. If $\Theta(E,h)\equiv 0$ then we have the following Griffiths formula:
\begin{equation}
(\Theta^{\mathcal H}_{t\bar t} [u], [u])= ||\mathbb H(\dbar V|_{X_t} \lrcorner ~\mathbb Hu)||^2-  ||\mathbb H(\partial \bar V|_{X_t} \lrcorner ~\mathbb Hu)||^2.
\end{equation}
If $p+q=n$, $E$ is a line bundle and $i\Theta(E, h)=\omega$ then
\begin{equation}
(\Theta^{\mathcal H}_{t\bar t} [u], [u]) \geq  (|V|^2_{\omega} \mathbb H u, \mathbb H u)+||\mathbb H(\dbar V|_{X_t} \lrcorner ~\mathbb Hu)||^2-  ||\mathbb H(\partial \bar V|_{X_t} \lrcorner ~\mathbb Hu)||^2.
\end{equation}
If $p+q=n$, $E$ is a line bundle and $i\Theta(E, h)=-\omega$  then
\begin{equation}
(\Theta^{\mathcal H}_{t\bar t} [u], [u]) \leq - (|V|^2_{\omega} \mathbb H u,\mathbb H u)+||\mathbb H(\dbar V|_{X_t} \lrcorner ~\mathbb Hu)||^2-  ||\mathbb H(\partial \bar V|_{X_t} \lrcorner ~\mathbb Hu)||^2.
\end{equation}
\end{theorem}

\end{document}